\documentclass[sn-mathphys-num]{sn-jnl}% Math and Physical Sciences Numbered Reference Style 
\usepackage{graphicx}%
\usepackage{multirow}%
\usepackage{amsmath,amssymb,amsfonts}%
\usepackage{amsthm}%
\usepackage{mathrsfs}%
\usepackage[title]{appendix}%
\usepackage{xcolor}%
\usepackage{textcomp}%
\usepackage{manyfoot}%
\usepackage{booktabs}%
\usepackage{algorithm}%
\usepackage{algorithmicx}%
\usepackage{algpseudocode}%
\usepackage{listings}%
\usepackage{mathtools}
\usepackage{tikz}
\usepackage{amscd} 
\usetikzlibrary{cd}
\usepackage{newtxmath,newtxtext}

\DeclareSymbolFont{boldlargesymbols}{LMX}{ntxexx}{b}{n}
\DeclareMathAccent{\bwidetilde}{\mathord}{boldlargesymbols}{"65}
\numberwithin{equation}{section}

\theoremstyle{thmstyleone}%
\newtheorem{theorem}{Theorem}[section]% meant for sectionwise numbers
\newtheorem{proposition}[theorem]{Proposition}% 
\newtheorem{lemma}[theorem]{Lemma}% 
\newtheorem{cor}[theorem]{Corollary}

\theoremstyle{definition}%
\newtheorem{definition}[theorem]{Definition}%
\newtheorem{example}[theorem]{Example}%
\newtheorem{remark}[theorem]{Remark}%

\begin{document}

\title[Geodesic Connectedness on Statistical Manifolds with Divisible Cubic Forms]{Geodesic Connectedness on Statistical Manifolds with Divisible Cubic Forms}


\author*[1]{\fnm{Ryu} \sur{Ueno}}

\affil*[1]{\orgdiv{Department of Mathematics}, 
\orgname{Hokkaido University}, 
\orgaddress{%\street{Street}, 
\city{Sapporo}, \postcode{060-0810}, 
%\state{State}, 
\country{Japan}}}


\abstract{A statistical manifold is just a manifold equipped with a Riemannian metric and a torsion-free affine connection that satisfies the Codazzi equation.
Although the term originates from information geometry, statistical manifolds have long existed in affine differential geometry.
The class of statistical manifolds with divisible cubic forms arises from affine differential geometry.
We examine the geodesic connectedness of affine connections on this class of statistical manifolds.
In information geometry, the geodesic connectedness of the affine connections are often assumed.
In Riemannian geometry, the geodesic connectedness of the Levi-Civita connection follows from its geodesic completeness by the well-known Hopf-Rinow theorem.
However, the geodesic connectedness of general affine connections is more challenging to achieve, even for the Levi-Civita connection in pseudo-Riemannian geometry or for affine connections on compact manifolds.
By analogy with the Hopf-Rinow theorem in Riemannian geometry, we establish the geodesic connectedness of the affine connections on statistical manifolds with divisible cubic forms from their geodesic completeness.
As an application, we establish a Cartan-Hadamard type theorem for statistical manifolds.}




\maketitle

\section{Introduction}\label{sec1}

We derive the geodesic connectedness of affine connections on statistical manifolds from its completeness for the first time.
% As an application, we also derive a statistical manifold version of the Cartan-Hadamard type theorem.
The term \textit{statistical manifold} originates from information geometry, but it has also appeared in affine differential geometry. 
A statistical structure on a manifold $M$ is defined by a pair $(g,\nabla)$, where $g$ is a Riemannian metric, and $\nabla$ is an affine connection on $M$.
This structure is in one-to-one correspondence with the pair $(g,C)$, where $C$ is a totally symmetric $(0,3)$-type tensor field on $M$, called the \textit{cubic form}.
In other words, a statistical structure can be considered as the pair of a cubic form and a Riemannian metric.\\

We refer to geodesic completeness of affine connections as simply \textit{completeness}. 
The completeness of the Levi-Civita connection on a Riemannian manifold is characterized by the Hopf-Rinow theorem. 
In particular, this theorem implies that if the Levi-Civita connection of a Riemannian manifold is complete, then it is \textit{geodesically connected}, that is, any two points on the manifold can be connected by a geodesic.
However, this does not hold for general affine connections. 
A well-known counterexample in pseudo-Riemannian geometry is the de Sitter space \cite{oneill1983semiriemannian}.
Moreover, even a complete affine connection on a compact manifold is not necessarily geodesically connected \cite{BATES1998273}.
The geodesic connectedness of an affine connection in the non-Riemannian setting has been widely studied from various perspectives \cite{SANCHEZ20013085, BARTOLO2011779, SIFJ, Beem1989PseudoconvexityAG}.
In information geometry, the existence of geodesics are often assumed, such as in the generalized Pythagorean theorem \cite{MR1800071} and constructing the canonical contrast function in \cite{Felice2018TowardsAC}. 
There are complete yet not geodesically connected affine connections in the geometry of statistical manifolds, as we demonstrate in Example \ref{countex}. 
However, to the author's knowledge, no research has been conducted on geodesic connectedness for the affine connections on statistical manifolds. \\

We establish the Hopf-Rinow type theorem for statistical manifolds with divisible cubic forms, that is, we derive the geodesic connectedness of the affine connection from its completeness.
A statistical manifold $(M,g,\nabla)$ is said to have \textit{divisible cubic form} $C$ if there exists some smooth function $\sigma$ on $M$ such that $C=\operatorname{sym}(d\sigma\otimes g)$.
Our main result is the following$:$\\

\textbf{Theorem}~\ref{main2}.
Assume that $\sigma$ is bounded from below on $M$. 
If $\nabla$ is complete, then it is geodesically connected.\\

Statistical manifolds with divisible cubic forms naturally arise in affine immersions whose image lies on a quadratic hypersurface.
The first study of such statistical manifolds was conducted by M. Noguchi \cite{NOGUCHI1992197}, where they were referred to as \textit{special statistical manifolds}. 
More recently, this statistical manifold has also appeared in the study of weighted Riemannian manifolds \cite{Li2016AnIF}.
However, since Noguchi's work, little research has been done on their deeper properties; they have only been considered as specific examples of statistical manifolds in \cite{BO, MIN201539, MeliCalvinB.2023Ssai}.
We reveal that statistical manifolds with divisible cubic forms possess many favorable properties.
In particular, significant progress was made in understanding the geodesics of the affine connection.\\

The structure of this paper is as follows. 
In \S 2, we provide a summary of fundamental concepts of statistical manifolds. 
In \S 3, we introduce the definition and properties of statistical manifolds with divisible cubic forms, where Subsection 3.1 mostly reviews known results, and Subsection 3.2 presents new findings. 
% Specifically, we show that the connection of a statistical manifold with divisible cubic form is projectively equivalent to the Levi-Civita connection of a certain Riemannian metric, and we prove the existence of a volume form and a canonical contrast function.
In \S 4, we prove the Hopf-Rinow type theorem for statistical manifolds with divisible cubic forms, along with examples from centro-affine differential geometry.
Finally, in \S 5 we apply these results to derive a statistical manifold version of the Cartan-Hadamard type theorem.\\

\section{Preliminaries}\label{sec2}

In this section, we organize definitions and notations in this paper. 
All objects in this paper are assumed to be smooth, and $M$ or $M^n$ denotes a connected orientable manifold of dimension $n\geq2$.
On a vector bundle $E$ over some manifold, we denote by $\Gamma(E)$ the set of sections of $E$.
\subsection{Statistical manifolds}
Let $g$ be a Riemannian metric on $M$, and $\nabla$ a torsion-free affine connection on $M$.
If the \textit{Cubic form} $C=\nabla g$ is totally symmetric on $M$, the triplet $(M,g,\nabla)$ is called a \textit{statistical manifold}.
Here, the pair $(g,\nabla)$ is called the \textit{statistical structure} on $M$.
The Levi-Civita connection of $g$ will be denoted as $\nabla^g$.
Another torsion-free affine connection $\overline{\nabla}$, called the \textit{conjugate connection} is naturally defined on the statistical manifold by the following equation:
\begin{equation*}
    Xg(Y,Z)=g(\nabla_X Y,Z)+g(Y,\overline{\nabla}_X Z),\quad X,Y,Z\in\Gamma(TM).
\end{equation*}
The triplet $(M,g,\overline{\nabla})$ is also a statistical manifold, and the Cubic form $\overline{C}$ satisfies $\overline{C} = -C$.\\

The difference tensor $K$ on $(M,g,\nabla)$ is defined by 
\begin{equation}
    \label{difftendef}
    K_X Y = K(X,Y) = \nabla_X Y - \nabla^g_X Y,\quad X,Y\in\Gamma(TM).
\end{equation}
The cubic form $C$ and the difference tensor $K$ are related by the following equation$:$
\begin{equation}
    \label{relcubdiff}
    C(X,Y,Z) = -2g(K_X Y,Z),\quad X,Y,Z\in\Gamma(TM).
\end{equation}
Therefore, the conditions $\nabla g = 0$ and $K=0$ are equivalent. 
In this case, the statistical manifold is called \textit{Riemannian}.\\

Conversely, given a totally symmetric $(0,3)$-tensor $C$ on a Riemannian manifold $(M,g)$, one can define a statistical manifold $(M,g,\nabla)$ using equations $(\ref{difftendef})$ and $(\ref{relcubdiff})$.
Therefore, the pair $(g,C)$ can be considered as the statistical structure.\\

For an affine connection $\nabla$ on $M$, we define the \textit{curvature tensor field} $R^\nabla \in \Gamma(TM^{(1,3)})$ 
of $\nabla$ by
\begin{equation*}
    R^\nabla (X,Y)Z=\nabla_X \nabla_Y Z-\nabla_Y \nabla_X Z-\nabla_{[X,Y]}Z,
\quad X, Y, Z \in \Gamma(TM). 
\end{equation*}
On a statistical manifold $(M,g,\nabla)$, we often denote $R^{\nabla}$ by $R$, 
$R^{\nabla^g}$ by $R^g$, and $R^{\overline{\nabla}}$ by $\overline{R}$, for convinience. 
For $X,Y,Z,W\in\Gamma(TM)$, the curvature tensor fields are related as follows:
\begin{eqnarray}
    &&
    g(R(X,Y)Z,W) = -g(Z,\overline{R}(X,Y)W), \label{intcurvature}\\
    &&
    R^g(X,Y)Z = R(X,Y)Z + (\nabla^g_X K)(Y,Z) - (\nabla^g_Y K)(X,Z) + \lbrack K_X,K_Y \rbrack Z, \label{Levicurvandstat}\\
    &&
    2R^g(X,Y)Z = R(X,Y)Z + \overline{R}(X,Y)Z + 2[K_X,K_Y]Z.\label{addcurv}
\end{eqnarray}
    If we have $R=\overline{R}$, then the statistical manifold is said to be \textit{conjugate symmetric}. \\
    \begin{remark}
        The conjugate symmetricity of $(M,g,\nabla)$ is equivalent to each of the following conditions$:$
        \renewcommand{\labelenumi}{$\operatorname{(\theenumi)}$}
        \begin{enumerate}
            \item $g(R(X,Y)Z,W) = -g(Z,R(X,Y)W), \quad X,Y,Z,W\in\Gamma(TM)$.
            \item $\nabla^gK$ is totally symmetric on $M$.
        \end{enumerate}
    \end{remark}
\noindent
We define the \textit{Ricci curvature tensor field} $\operatorname{Ric}^{\nabla}\in\Gamma(TM^{(0,2)})$ of $\nabla$ by
    \begin{equation*}
        \operatorname{Ric}^{\nabla}(X,Y)=\sum_{i=1}^ng(R^{\nabla}(e_i,X)Y,e_i),\quad X,Y\in\Gamma(TM), 
    \end{equation*}
where $\{e_1,\dots,e_n\}$ is an orthonormal frame on $M$ with respect to $g$. 
On a statistical manifold $(M,g,\nabla)$, we often denote $\operatorname{Ric}^{\nabla}$ by $\operatorname{Ric}$, 
$\operatorname{Ric}^{\nabla^g}$ by $\operatorname{Ric}^g$, and $\operatorname{Ric}^{\overline{\nabla}}$ by $\overline{\operatorname{Ric}}$ for convinience. \\
\begin{definition}
    A statistical manifold $(M,g,\nabla)$ is said to be of \textit{constant sectional curvature} $\lambda$ if the following equation hold for some real number $\lambda$$:$
    \begin{equation*}
        R(X,Y)Z=\lambda(g(Y,Z)X-g(X,Z)Y),\quad X,Y,Z\in\Gamma(TM).
    \end{equation*}
\end{definition}

The next proposition is known for statistical manifolds of constant sectional curvature \cite{MR4452143}.\\
\begin{proposition}
    Let $(M,g,\nabla)$ be a statistical manifold. The following conditions are mutually equivalent$:$
    \renewcommand{\labelenumi}{$\operatorname{(\theenumi)}$}
    \begin{enumerate}
        \item It has constant sectional curvature.
        \item It is conjugate symmeteric and $\nabla$ is projectively flat.
    \end{enumerate}
\end{proposition}

Statistical manifolds were originally found in affine differential geometry. 
The details of the next example can be found in \cite{MR2648267} and \cite{NomizuSasaki}.

\vspace{0.5\baselineskip}

\begin{example}
    \label{affin}
    Let $f:M^n\to\mathbb{R}^{n+1}$ be an immersed hypersurface, and $\xi\in\Gamma(f^*T\mathbb{R}^{n+1})$ an transversal vector field of $f(M)$. 
    The pair of $\{f,\xi\}$ is called an \emph{affine immersion}.
    A torsion-free affine connection $\nabla$ and a symmetric $(0,2)$-tensor field $g$ are induced by the Gauss formula$:$
    \begin{equation}
        \label{Gaussfor}
        D_X f_*Y = f_*\nabla_X Y + g(X,Y)\xi,\quad X,Y\in\Gamma(TM)
    \end{equation}
    The immersion $f$ is called \textit{locally strongly convex}, if the tensor field $g$ is definite. 
    We will redefine $\xi$ if necessary so that $g$ is positive definite.
    Define the induced volume form $\omega_{\xi}$ on $M$ by
    \begin{equation}
        \label{volform}
        \omega_{\xi}(X_1,\dots,X_n) = \det(f_*X_1,\dots,f_*X_n,\xi),\quad X_1,\dots,X_n\in\Gamma(TM).
    \end{equation}
    The affine immersion $\{f,\xi\}$ is called an \textit{equiaffine immersion} if $\nabla \omega_{\xi}=0$ holds.
    This is equivalent to the total symmetry of the cubic form $C=\nabla g$.
    The induced structure $(g,\nabla)$ by a locally strongly convex immersion $f$ and an equiaffine transversal vector field $\xi$ is a statistical structure on $M$.
    
    If $\xi$ is a constant vector, then $(M,g,\nabla)$ has constant sectional curvature $0$. 
    If $\xi=kf$ holds for a non-zero real number $k$, then $(M,g,\nabla)$ is of constant sectional curvature $-k$.
    Particularly, if $\xi=\pm f$ holds, then the affine immersion is called a \textit{centroaffine immersion}.

    Let $(g,\nabla)$ be a statistical structure $M$ induced from an affine immersion $\{f,\xi\}$.
    The conormal map $\nu:M\to(\mathbb{R}^{n+1})^*$ of $\{f,\xi\}$ is defined by
    \begin{equation*}
        \nu_x(f_*T_xM)=0,\quad \nu_x(\xi_x)=1,\quad x\in M.
    \end{equation*}
    For each $x$, the vector $\nu_x\in(\mathbb{R}^{n+1})^*$ is transversal to $\nu_{*x}(T_xM)$.
    The Gauss formula of the centroaffine immersion $\nu$ is
    \begin{equation*}
        D_X \nu_*Y = \nu_*\overline{\nabla}_X Y-\frac{1}{n-1}\overline{\operatorname{Ric}}(X,Y)\nu, \quad X,Y\in\Gamma(TM).
    \end{equation*}
    Here, $\overline{\nabla}$ is the conjugate connection of $(g,\nabla)$.\\
\end{example}

\vspace{0.5\baselineskip}

The fundamental theorem of affine differential geometry characterizes when statistical structures can be realized from an affine connection \cite{NomizuSasaki}.

\vspace{0.5\baselineskip}

\begin{theorem}
    Let $(g,\nabla)$ be a statistical structure on a simply connected manifold $M$.
    There exists a locally strongly convex equiaffine immersion $\{f,\xi\}$ that induces $(g,\nabla)$ on $M$ by the Gauss formula, if and only if $\overline{\nabla}$ is projectively flat and $\operatorname{Ric}$ is symmetric.\\
\end{theorem}

\subsection{Contrast functions}
A \textit{contrast function} is a function that defines a statistical structure on a manifold.\\
\begin{definition}
    Let $M$ be a smooth manifold. A contrast function $\rho:M\times M\to\mathbb{R}$ of $M$ is a smooth function that satisfies the following conditions:
    \renewcommand{\labelenumi}{$\operatorname{(\theenumi)}$}
    \begin{enumerate}
        \item For any $p,q\in M$, $\rho(p,q)\geq0$. Particularly, $\rho(p,q)=0$ if and only if $p=q$.
        \item For any $p\in M$ and any non-zero $X\in T_pM$, $(X^LX^R\rho)(p,p)<0$ holds.
    \end{enumerate}
    Here, $X^L=(X,0),X^R=(0,X)\in T_{(p,p)}(M\times M)$.\\
\end{definition}

For $X_1,\dots,X_l,Y_1,\dots,Y_r\in\Gamma(TM)$ and $p\in M$, we use the notation
\begin{equation*}
    \rho(X_1,\dots,X_l|Y_1,\dots,Y_r)(p)=(X_1^L\cdot\cdot\cdot X_l^L Y_1^R\cdot\cdot\cdot Y_r^R\rho)(p,p).
\end{equation*}
For any fixed $p\in M$, the function $\rho(p,q),\, q\in M$ attains a global minimum at $p=q$.
Thus, we have $\rho(X|)=\rho(|X)=0$ for any $X\in\Gamma(TM)$. 
We also have
\begin{equation*}
    0=X\rho(Y|)=\rho(XY|)+\rho(X|Y)
\end{equation*}
For any $X,Y\in\Gamma(TM)$.

Let us induce a statistical structure $(g,\nabla)$ on $M$ by the contrast function.
First we define the Riemannian metric $g$ on $M$ by
\begin{equation*}
    g(X,Y) = -\rho(X|Y), \quad X,Y\in\Gamma(TM).
\end{equation*}
The affine connection $\nabla$ on $M$ is induced by
\begin{equation*}
    g(\nabla_X Y, Z) = -\rho(XY|Z), \quad X,Y,Z\in\Gamma(TM).
\end{equation*}
The statistical structure induced by a contrast function has the following properties.
See \cite{10.32917/hmj/1206128508} for a proof.
\begin{enumerate}
    \item $g(\overline{\nabla}_X Y,Z) = -\rho(Z|XY)$,
    \item $(\nabla_X g)(Y,Z) = \rho(XY|Z) - \rho(Z|XY)$,
\end{enumerate}
where $X,Y,Z\in\Gamma(TM)$.\\

It is known that any statistical manifold has a contrast function that induces the original statistical structure \cite{10.32917/hmj/1206128255}, but it is difficult to find the most suitable one.
The canonical contrast function of statistical structures induced by an affine immersion is the next example given in \cite{Kurose1994ONTD}.\\
\begin{example}
    \label{ccf1}
    Let $(g,\nabla)$ be a statistical structure on $M$ induced by a locally strongly convex equiaffine immersion $\{f,\xi\}$ as in Example \ref{affin}.
    Let $\nu:M\to\mathbb({R}^{n+1})^*$ be the conormal map of $\{f,\xi\}$. 
    A contrast function $\rho$ on $(M,g,\nabla)$ is defined by the following equation$:$
    \begin{equation}
        \rho(p,q)=\nu_q(f(p)-f(q)),\quad p,q\in M.
    \end{equation} 
    Note that the definition of $\rho$ is independent of the choice of the affine immersion that induces $(g,\nabla)$.
\end{example}

\section{Statistical manifolds with divisible cubic forms}
We derive the class of statistical manifolds with divisible cubic forms from affine differential geometry.
The simplest affine immersions are the one that their images lies on a hyperquadric.
The statistical structure induced by an affine immersion to some hyperquadric is characterized as below \cite{NomizuSasaki}.\\
\begin{proposition}
    Let $(g,\nabla)$ be a statistical structure on $M$ induced by a locally strongly convex equiaffine immersion $\{f,\xi\}$.
    The image of $f:M^n\to\mathbb{R}^{n+1}$ lies on a hyperquadric in $\mathbb{R}^{n+1}$ if and only if there exists $\sigma\in C^{\infty}(M)$ such that 
    \begin{align}
        C(X,Y,Z) &= \operatorname{sym}(d\sigma\otimes g)(X,Y,Z) \notag\\
                 &= d\sigma(X)g(Y,Z) + d\sigma(Y)g(Z,X) + d\sigma(Z)g(X,Y), \quad X,Y,Z\in\Gamma(TM).\label{divcub}
    \end{align}
\end{proposition}

\begin{definition}
    Let $(M,g,\nabla)$ be a statistical manifold. The statistical manifold is said have \textit{divivisible cubic form} if the cubic form satisfies the equation $(\ref{divcub})$ for some function on $M$.\\
\end{definition}

On a statistical manifold $(M,g,\nabla)$ with $C=\operatorname{sym}(d\sigma\otimes g)$, the difference tensor $K$ is given by the following equation$:$
\begin{equation}
    \label{difften}
    K_X Y = -\frac{1}{2}(d\sigma(X)Y + d\sigma(Y)X + g(X,Y)\operatorname{grad}_g\sigma).
\end{equation}
Here, $X,Y\in\Gamma(TM)$ and $\operatorname{grad}_g\sigma$ is the gradient vector field of $\sigma$ with respect to $g$.
The divisibility of the cubic form can be rewritten by the difference tensor satisfying equation $(\ref{difften})$ for some smooth function.\\

\begin{example}
    \label{ellipex}
    Let $f:\mathbb{R}^2\to\mathbb{R}^3$ be an elliptic paraboloid, that is$:$
    \begin{equation*}
        f(x^1,x^2) = \left(x^1,x^2,\frac{1}{2}((x^1)^2+(x^2)^2-1)\right),\quad (x^1,x^2)\in\mathbb{R}^2,
    \end{equation*}
    for some coordinate system $\{y^1,y^2,y^3\}$ on $\mathbb{R}^3$.
    The immersion $f$ is a centroaffine immersion.
    The statistical structure induced by the Gauss formula $(\ref{Gaussfor})$ with the transversal vector field $-f$ is
    \begin{align}
        g &= \frac{2}{(x^1)^2+(x^2)^2+1}\left((dx^1)^2 + (dx^2)^2\right), \notag\\
        \nabla_{\frac{\partial}{\partial x^i}}\frac{\partial}{\partial x^j} =& \frac{2\delta_{ij}}{(x^1)^2+(x^2)^2+1}\left(x^1\frac{\partial}{\partial x^1}+x^2\frac{\partial}{\partial x^2}\right). \notag
    \end{align}
    Here, if we define $\sigma\in C^{\infty}(\mathbb{R}^2)$ by $\sigma(x^1,x^2)=-\log(\frac{1}{2}((x^1)^2 + (x^2)^2+1)),\,(x^1,x^2)\in\mathbb{R}^2$, then we have $\nabla g = \operatorname{sym}(d\sigma\otimes g)$.\\
    
    The conjugate connection $\overline{\nabla}$ of $(g,\nabla)$ is induced by the conormal map of $\{f,-f\}$$:$
    \begin{equation*}
        \nu_{(x^1,x^2)} = \frac{1}{(x^1)^2+(x^2)^2+1}\left(-2x^1dy^1-2x^2dy^2+dy^3\right)\in\mathbb{R}^3,\quad (x^1,x^2)\in\mathbb{R}^2
    \end{equation*}
    and it is$:$
    \begin{equation*}
        \overline{\nabla}_{\frac{\partial}{\partial x^i}}\frac{\partial}{\partial x^j} = -\frac{2}{(x^1)^2+(x^2)^2+1}\left(x^i\frac{\partial}{\partial x^j}+x^j\frac{\partial}{\partial x^i}\right). \notag
    \end{equation*}
    The statistical structure $(g,\nabla)$ is induced from a centroaffine immersion, and the statistical manifold $(\mathbb{R}^2,g,\nabla)$ is of constant sectional curvature $1$.
    Thus, we have $\operatorname{Ric}=\overline{\operatorname{Ric}}=g$.\\
\end{example}

The class of statistical manifolds with divisible cubics forms can also be characterized using the following equivalence of statistical structures, introduced in \cite{Matsuzoe1998OnRO}.\\

\begin{definition}
    Let $(g,\nabla)$ and $(\tilde{g},\tilde{\nabla})$ be two statistical structures on a manifold $M$.
    The two statistical structures are \textit{conformally-projectively equivalent} if there exists $\varphi,\psi\in C^{\infty}(M)$ such that
    \begin{align}
        \tilde{g} &= e^{\varphi+\psi}g \label{cp1}, \\
        \tilde{\nabla}_X Y &= \nabla_X Y + d\varphi(X)Y + d\varphi(Y)X - g(X,Y)\operatorname{grad}_g\psi, \label{cp2}
    \end{align}
    where $X,Y\in\Gamma(TM)$.\\
\end{definition}

\begin{proposition}
    A statistical manifold $(M,g,\nabla)$ has a divisible cubic form if and only if there exists a Riemannian metric $\tilde{g}=e^{\sigma}g$ on $M$ such that the statistical structures $(g,\nabla)$ and $(\tilde{g},\nabla^{\tilde{g}})$ are conformally-projectively equivalent.\\
\end{proposition}
\begin{proof}
    First, let $(M,g,\nabla)$ be a statistical manifold with divisible cubic form of $C=\operatorname{sym}(d\sigma\otimes g)$.
    The Levi-Civita connections of $g$ and $\tilde{g}=e^{\sigma}g$ are related by 
    \begin{equation}
        \label{contrans}
        \nabla^{\tilde{g}}_X Y = \nabla^g_X Y + \frac{1}{2}\left(d\sigma(X)Y+d\sigma(Y)X-g(X,Y)\operatorname{grad}_g\sigma\right)
    \end{equation}
    for $X,Y\in\Gamma(TM)$. Therefore, combining with equation $(\ref{difften})$, the affine connections $\nabla$ and $\nabla^{\tilde{g}}$ are related by
    \begin{equation}
        \label{proj}
        \nabla^{\tilde{g}}_X Y = \nabla_X Y + d\sigma(X)Y+d\sigma(Y)X
    \end{equation}
    for $X,Y\in\Gamma(TM)$, which is the equation $(\ref{cp2})$ for $\varphi = \sigma$ and $\psi=0$.

    Conversely, suppose that on a statistical manifold $(M,g,\nabla)$, there exists a Riemannian metric $\tilde{g}$ on $M$ such that equations $(\ref{cp1})$ and $(\ref{cp2})$ hold for some $\varphi, \psi\in C^{\infty}(M)$.
    Then, difference tensor $K$ of $(M,g,\nabla)$ is given by
    \begin{equation*}
        K_X Y = -\frac{1}{2}\left(d(\varphi-\psi)(X)Y + d(\varphi-\psi)(Y)X + g(X,Y)\operatorname{grad}_g(\varphi-\psi)\right)
    \end{equation*}
    for $X,Y\in\Gamma(TM)$.
    Therefore, the cubic form of $(M,g,\nabla)$ is $C=\operatorname{sym}(d(\varphi-\psi)\otimes g)$.\\
\end{proof}

\subsection{Properties of statistical manifolds with divisible cubic forms}
Let $(M,g,\nabla)$ be a statistical manifold with $C=\operatorname{sym}(d\sigma\otimes g)$.
The covariant derivatives with respect to $\nabla^g$ of $C$ and $K$ are given by
\begin{align}
    (\nabla^g_X K)(Y,Z) &= -\frac{1}{2}\left((\nabla^g_X d\sigma)(Y)Z + (\nabla^g_X d\sigma)(Z)Y + g(Y,Z)\nabla^g_X \operatorname{grad}_g\sigma\right), \notag\\
    (\nabla^g_X C)(Y,Z,W) &= \operatorname{sym}((\nabla^g_X d\sigma)\otimes g)(Y,Z,W),\notag
\end{align}
where $X,Y,Z,W\in\Gamma(TM)$.\\

The curvature tensor fields and Ricci curvature tensor fields are given as follows in \cite{NOGUCHI1992197}.\\
\begin{proposition}
Let $(M^n,g,\nabla)$ be a statistical manifold with divisible cubic form of $C=\operatorname{sym}(d\sigma\otimes g)$.
We have that
\begin{align}
    R(X,Y)Z &= \frac{1}{2}\left((\nabla^g_X d\sigma)(Z)Y-(\nabla^g_Y d\sigma)(Z)X+g(Y,Z)\nabla^g_X\operatorname{grad}_g\sigma-g(X,Z)\nabla^g_Y\operatorname{grad}_g\sigma\right)\notag\\
    &+ R^g(X,Y)Z + \frac{1}{4}\left(d\sigma(Y)d\sigma(Z)X-d\sigma(X)d\sigma(Z)Y\right) \notag\\
    &+ \frac{1}{4}\left(g(Y,Z)(\|d\sigma\|_g^2X+d\sigma(X)\operatorname{grad}_g\sigma)-g(X,Z)(\|d\sigma\|_g^2Y+d\sigma(Y)\operatorname{grad}_g\sigma)\right),\\
    \operatorname{Ric}(X,Y) &= \operatorname{Ric}^g(X,Y) + \frac{1}{2}\left((n+2)(\nabla^g_X d\sigma)(Y) - g(X,Y)\Delta_g\sigma\right)\notag\\
    &+ \frac{1}{4}\left((n-2)d\sigma(X)d\sigma(Y) + n\|d\sigma\|_g^2g(Y,Z)\right),
\end{align}
where $X,Y,Z\in\Gamma(TM)$, $\|\cdot\|_g$ denotes the norm defined by $g$, and $\Delta_g$ is the Laplace operator on $(M,g)$ of negative eigenvalues.\\
\end{proposition}

The conjugate symmetry of a statistical manifold with divisible cubic form is given by the following proposition, see \cite{MIN201539} and \cite{NOGUCHI1992197} for proof.\\

\begin{proposition}
    \label{conj}
    Let $(M,g,\nabla)$ be a statistical manifold of divisible cubic form, where $C=\operatorname{sym}(d\sigma\otimes g)$.
    The following conditions are equivalent.
    \renewcommand{\labelenumi}{$\operatorname{(\theenumi)}$}
    \begin{enumerate}
        \item The statistical manifold $(M,g,\nabla)$ is conjugate symmetric.
        \item There exists $f\in C^{\infty}(M)$ such that $\nabla^gd\sigma = f\cdot g$ holds. \label{a}
        \item $\operatorname{Ric} = \overline{\operatorname{Ric}}$.\\
    \end{enumerate}
\end{proposition}

If condition $(\ref{a})$ in Proposition \ref{conj} holds on a statistical manifold $(M,g,\nabla)$ with $C=\operatorname{sym}(d\sigma\otimes g)$, the vector field $\operatorname{grad}_g\sigma$ is a concircular vector field on the Riemannian manifold $(M,g)$ \cite{math8040469}.
Particularly, we have $f=\frac{1}{n}\Delta_g\sigma$.
Indeed, if we let $\{e_1,\dots,e_n\}$ be an orthonormal frame of $(M,g)$, we have
\begin{equation*}
        n\cdot f= \sum_{i=1}^nf\cdot g(e_i,e_i) = \sum_{i=1}^n(\nabla^g_{e_i}d\sigma)(e_i) = \Delta_g\sigma.\\
\end{equation*}

\begin{example}
    Consider the statistical structure $(g,\nabla)$ on $\mathbb{R}^2$ induced from a centroaffine immersion in Example \ref{ellipex}.
    The statistical manifold $(M,g,\nabla)$ is of constant sectional curvature, hence it is conjugate symmetric.
    Since $g^*=e^{-\sigma}g$ is the Euclidean metric on $\mathbb{R}^2$, let us calculate $\nabla^gd\sigma$ by using the equation \ref{contrans}.
    We have
    \begin{equation*}
            (\nabla^g_Xd\sigma)(Y) = (\nabla^{g^*}_Xd\sigma)(Y) - d\sigma(X)d\sigma(Y) + \frac{1}{2}g^*(X,Y)\|d\sigma\|^2_{g^*}
    \end{equation*}
    for $X,Y\in\Gamma(TM)$.
    For $(X_1,X_2)=(\frac{\partial}{\partial x^1},\frac{\partial}{\partial x^2})$ and $(x^1,x^2)\in\mathbb{R}^2$, we have
    \begin{equation*}
        d\sigma(X_i)(x^1,x^2) = -\frac{2x^i}{(x^1)^2+(x^2)^2+1},
    \end{equation*}
    \begin{equation*}
        (\nabla^{g^*}_{X_i}d\sigma)(X_j)(x^1,x^2) = -\frac{2\delta_{ij}((x^1)^2+(x^2)^2+1)-4x^ix^j}{((x^1)^2+(x^2)^2+1)^2},
    \end{equation*}
    and
    \begin{equation*}
        g^*(X_i,X_j)\|d\sigma\|_{g^*}^2(x^1,x^2) = 4\delta_{ij}\frac{(x^1)^2+(x^2)^2}{((x^1)^2+(x^2)^2+1)^2}.
    \end{equation*}
    We have 
    \begin{equation*}
        (\nabla^g_{X_i}d\sigma)(X_j)(x^1,x^2) = -\frac{2\delta_{ij}}{((x^1)^2+(x^2)^2+1)^2}
    \end{equation*}
    for $(x^1,x^2)\in\mathbb{R}^2$.
    We also have
    \begin{equation*}
        (\Delta_g \sigma)(x^1,x^2) = -\frac{2}{(x^1)^2+(x^2)^2+1},\quad (x^1,x^2)\in\mathbb{R}^2,
    \end{equation*}
    which implies $\nabla^g d\sigma = \frac{1}{2}\Delta_g\sigma\cdot g$.
    Since the statistical manifold $(M,g,\nabla)$ is conjugate symmetric, we also obtain
    \begin{equation}
        \label{conjcovc}
        (\nabla^g_X C)(Y,Z,W) = \frac{1}{2}\Delta_g\sigma \operatorname{sym}(g(X,\cdot)\otimes g)(Y,Z,W)
    \end{equation}
    for each $X,Y,Z,W\in\Gamma(TM)$.\\
\end{example}

\begin{theorem}
    Let $(M,g,\nabla)$ be a compact statistical manifold with divisible cubic form.
    If $(M,g,\nabla)$ is non-Riemannian conjugate symmetric, and $(M,g)$ is an Einstein manifold, then $(M,g)$ must be isomorphic to the Euclidean n-sphere $S^n(c)$ of some constant curvature $c>0$.\\
\end{theorem}
\begin{proof}
    Let $(M,g,\nabla)$ be a compact statistical manifold with $C=\operatorname{sym}(d\sigma\otimes g)$, and let $\operatorname{Ric}^g = c\cdot g$ hold.
    By Proposition \ref{conj} $({\ref{a}})$ we have
    \begin{equation}
        n\cdot\nabla^g_X\operatorname{grad}_g\sigma = \Delta_g\sigma \cdot X,\quad X\in\Gamma(TM).
    \end{equation}
    Assume that $\Delta_g\sigma=0$. 
    Then, since $M$ is compact we obtain that $\sigma$ is a constant, which implies that $(M,g,\nabla)$ is Riemannian.
    Hence, $\Delta_g\sigma$ is not equal to $0$. 

    If there is a constant $\alpha$ such that $d\sigma = \alpha\cdot d\Delta_g\sigma$, Theorem 3.2 in \cite{math8040469} completes the proof.
    If we let $\{e_1,\dots,e_n\}$ be an orthonormal frame of $(M,g)$, for any $X\in\Gamma(TM)$ we have
    \begin{equation*}
        \begin{split}
            d\Delta_g\sigma (X) &= g\left(\operatorname{grad}_g\Delta_g\sigma,X\right)\\
            &= \sum_{i=1}^ng\left(\nabla^g_{e_i}\nabla^g_{e_i}\operatorname{grad}_g\sigma-\nabla^g_{\nabla^g_{e_i}e_i}\operatorname{grad}_g\sigma,X\right) - \operatorname{Ric}^g(\operatorname{grad}_g\sigma,X)\\
            &= \sum_{i=1}^ng\left(\frac{1}{n}e_i\left(\Delta_g \sigma\right) e_i,X\right) - c\cdot g(\operatorname{grad}_g\sigma,X)\\
            &= \frac{1}{n}d\Delta_g\sigma(X) - c\cdot d\sigma(X).
        \end{split}
    \end{equation*}
\end{proof}

An example of a statistical manifold on the Euclidean sphere with divisible cubic form is given in \cite{NOGUCHI1992197}, the first example of page 215.\\

An affine connection $\nabla$ on $M$ is called \textit{complete} if any geodesic of $\nabla$ can always be extended as a geodesic to have a parameter from $-\infty$ to $\infty$.
This is also equivalent to the exponential map $\operatorname{exp}_p$ of $\nabla$ being defined on all of $T_pM$, for any $p\in M$.
A Riemannian manifold is called complete if its Levi-Civita connection is complete.
The following proposition derives the completeness of an affine connection on a complete Riemannian manifold, including the ones on statistical manifolds with divisible cubic forms.
See \cite{BO} for proof.\\

\begin{proposition}
    \label{compconns}
    Let $(M, g)$ be a complete Riemannian manifold and \(\nabla\) an affine connection on $M$. 
    Suppose there exists a smooth function $\sigma \in C^{\infty}(M)$, bounded from below on $M$, such that for all vector fields $X, Y\in \Gamma(TM)$, the equation
    \begin{equation*}
        \nabla_X Y = \nabla^g_X Y + \alpha_1 d\sigma(X)Y
        + \alpha_2 d\sigma(Y)X
        + \alpha_3 g(X, Y)\operatorname{grad}_g\sigma
    \end{equation*}
    holds, where $\alpha_1, \alpha_2, \alpha_3$ are real constants satisfying $\alpha_1 + \alpha_2 + \alpha_3 \geq 0$.
    Then, the affine connection $\nabla$ is complete.\\
\end{proposition}

\subsection{The projectively equivalent Levi-Civita connection}
Recall that two affine connections $\nabla$ and $\tilde{\nabla}$ are said to be \textit{projectively equivalent} if there exists a closed form $\lambda$ such that
\begin{equation*}
    \tilde{\nabla}_XY= \nabla_XY + \lambda(X)Y + \lambda(Y)X, \quad X,Y\in\Gamma(TM).
\end{equation*}
An affine connection is called \textit{projectively flat} if it is projectively equivalent to a flat connection around any point of the manifold.
From equation $(\ref{proj})$, the Levi-Civita connection of $\tilde{g}=e^{\sigma}g$ and $\nabla$ are projectively equivalent on a statistical manifold $(M,g,\nabla)$ with divisible cubic form $C=\operatorname{sym}(d\sigma \otimes g)$.
If the Levi-Civita connection on a Riemannian manifold is projectively flat, then the Riemannian manifold has constant curvature.
Thus, we obtain the following theorem.\\

\begin{theorem}
    Let $(M,g,\nabla)$ be a statistical manifold with divisible cubic form $C=\operatorname{sym}(d\sigma\otimes g)$ and $\tilde{g}=e^{\sigma}g$.
    The affine connection $\nabla$ is projectively flat if and only if $(M,\tilde{g})$ is of constant curvature.\\
\end{theorem}

The notion of statistical curvature on a statistical manifold is introduced in \cite{SHS}.\\
\begin{definition}
    The \textit{statistical curvature tensor field} $S$ of a statistical manifold $(M,g,\nabla)$ is defined by\\
    \begin{equation*}
        S = \frac{1}{2}(R+\overline{R}).
    \end{equation*}
    In particular, we have $S=R$ if $(M,g,\nabla)$ is conjugate symmetric.\\
\end{definition}

Let $\tilde{R}$ be the curvature tensor field of $\nabla^{\tilde{g}}$.
The following equations hold for $X,Y,Z\in\Gamma(TM)$$:$
\begin{equation*}
    \begin{split}
        R(X,Y)Z  &= \tilde{R}(X,Y)Z + \frac{1}{2}\|d\sigma\|_g^2\left(g(Y,Z)X-g(X,Z)Y\right)\\
        & - (\nabla^g_Xd\sigma)(Z)Y + (\nabla^g_Yd\sigma)(Z)X,\\
        \overline{R}(X,Y)Z &= \tilde{R}(X,Y)Z + \frac{1}{2}\|d\sigma\|_g^2\left(g(Y,Z)X-g(X,Z)Y\right)\\
                    & - g(X,Z)\nabla^g_Y\operatorname{grad}_g\sigma + g(Y,Z)\nabla^g_X\operatorname{grad}_g\sigma.
    \end{split}
\end{equation*}
For each pair of orthonormal vectors $X,Y$ of $(M,g)$, we have 
\begin{equation}
    \label{seccurvtilde}
        \frac{\tilde{g}\left(\tilde{R}(X,Y)Y,X\right)}{\tilde{g}(X,X)\tilde{g}(Y,Y)-\tilde{g}(X,Y)^2} = e^{\sigma}\left(g(S(X,Y)Y,X) - \frac{1}{2}\left((\nabla^g_Xd\sigma)(X) + (\nabla^g_Yd\sigma)(Y) + \|d\sigma\|_g^2\right)\right).
\end{equation}
Therefore, the affine connection $\nabla$ is projectively flat if and only if the right-hand side of equation $(\ref{seccurvtilde})$ is constant for all pairs of orthonormal vectors on $(M,g)$.\\

The projective equivalence between $\nabla$ and $\nabla^{\tilde{g}}$ lets us investigate the geodesics of $\nabla$ much easier than general affine connections.
This is because a geodesic of $\nabla$ becomes a geodesic of $\nabla^{\tilde{g}}$ under some parameter transformation.
Indeed, let $\gamma(t),\,t\in[0,a)$ be a geodesic of $\nabla$. 
Define $s:[0,a)\to[0,\alpha)$ as
\begin{equation}
    \label{paramtrans}
    s(t) = \int_0^t e^{2\sigma(\gamma(r))} dr,\quad t\in[0,a).
\end{equation}
Here, we obtain
\begin{equation*}
    \frac{ds}{dt}(t) = e^{2\sigma(\gamma(t))}>0, \quad t\in[0,a),
\end{equation*}
thus, the parameter transformation strictly increases.
We define a parametrized curve by $\tilde{\gamma}=\gamma\circ s^{-1}$.
Since we have
\begin{equation*}
    \frac{dt}{ds} = e^{-2\sigma(\gamma\circ s^{-1})},
\end{equation*}
the velocity of $\tilde{\gamma}$ is
\begin{equation*}
    \frac{d\tilde{\gamma}}{ds} = e^{-2\sigma(\tilde{\gamma})}\frac{d\gamma}{dt}\circ s^{-1}.
\end{equation*}
Consequently, the next calculation shows that $\tilde{\gamma}$ is a geodesic of $\nabla^{\tilde{g}}$$:$
\begin{equation*}
    \begin{split}
        \nabla^{\tilde{g}}_{\frac{d\tilde{\gamma}}{ds}}\frac{d\tilde{\gamma}}{ds}&= -2\frac{d}{ds}(\sigma(\tilde{\gamma}))e^{-2\sigma(\tilde{\gamma})}\frac{d\gamma}{dt}\circ s^{-1}+e^{-4\sigma(\tilde{\gamma})}\nabla^{\tilde{g}}_{\frac{d\gamma}{dt}}\frac{d\gamma}{dt}\circ s^{-1}\\
        &=-2d\sigma\left(\frac{d\tilde{\gamma}}{ds}\right)\frac{d\tilde{\gamma}}{ds} + 2e^{-4\sigma(\tilde{\gamma})}d\sigma\left(\frac{d\gamma}{dt}\circ s^{-1}\right)\frac{d\gamma}{dt}\circ s^{-1}\\
        &=0.
    \end{split}
\end{equation*}
Here, the projective equivalence $(\ref{proj})$ is used in the second line.\\

\subsection{The volume form and canonical contrast function}
In this subsection, we define the volume form and the canonical contrast function on statistical manifolds with divisible cubic forms.

The following theorem provides an existence of a volume form parallel with respect to the affine connection on a statistical manifold with divisible cubic form.
Note that this does not necessarily hold for general statistical manifolds.
If the statistical structure is induced by an affine immersion, the volume form coincides with that defined by $(\ref{volform})$.\\

\begin{theorem}
    Let $(M^n,g,\nabla)$ be a statistical manifold with divisible cubic form $C=\operatorname{sym}(d\sigma\otimes g)$.
    Then, the volume form $\theta = e^{-\frac{n+2}{2}\sigma}\cdot\omega_g$ is parallel with respect to $\nabla$, where $\omega_g$ is the standard volume form on $(M,g)$.\\
\end{theorem}
\begin{proof}
    Let $\{e_1,\dots,e_n\}$ be an orthonormal frame of $(M,g)$. 
    For any $X\in\Gamma(TM)$, we have
    \begin{equation*}
        \begin{split}
            \mathrm{tr}K_X &= \sum_{i=1}^{n}g\left(K_X e_i,e_i\right)\\
            &=-\frac{1}{2}\sum_{i=1}^{n}\operatorname{sym}(d\sigma\otimes g)\left(X,e_i,e_i\right)\\
            &=-\frac{n+2}{2}d\sigma (X).
        \end{split}
    \end{equation*}
    Therefore, if we let $\theta=e^{-\frac{n+2}{2}\sigma}\omega_g$, for any $X\in\Gamma(TM)$ we have
    \begin{equation*}
        \begin{split}
            \nabla_X \theta &= d\left(e^{-\frac{n+2}{2}\sigma}\right)(X)\cdot\omega_g + e^{-\frac{n+2}{2}\sigma}\nabla_X\omega_g\\
            % &= e^{-\frac{n+2}{2}\sigma}\left(-\frac{n+2}{2}d\sigma(X)\cdot\omega_g + K_X\omega_g\right)\\
            &= e^{-\frac{n+2}{2}\sigma}\left(-\frac{n+2}{2}d\sigma(X)\cdot\omega_g - \mathrm{tr}K_X\cdot\omega_g\right) = 0.
        \end{split}
    \end{equation*}
\end{proof}

\begin{remark}
    The Ricci curvature tensor field of an affine connection $\nabla$ is symmetric if and only if there exists a volume form $\theta$ around any point on the manifold such that $\nabla \theta = 0$.
    Thus, the Ricci curvature tensor field of the affine connection on a statistical manifold with divisible cubic form is always symmetric.\\
\end{remark}

The canonical contrast function on statistical manifolds with divisible cubic forms is provided by the following theorem.\\

\begin{theorem}
    \label{contfunc}
  Let $(M,g,\nabla)$ be a statistical manifold with divisible cubic form with $C=\operatorname{sym}(d\sigma\otimes g)$ and let $\tilde{g}=e^{\sigma}g$.
  The function $\rho:M\times M\to\mathbb{R}$ defined by
  \begin{equation*}
    \rho(p,q) = e^{-\sigma(p)}(\tilde{d}(p,q))^2,\quad p,q\in M
  \end{equation*}
  is a contrast function on $(M,g,\nabla)$. Here, $\tilde{d}$ is the distance function of $(M,\tilde{g})$.\\
\end{theorem}
\begin{proof}
    Note that $(\tilde{d})^2$ is a contrast function of the Riemannian statistical manifold $(M,\tilde{g},\nabla^{\tilde{g}})$.
    It is clear that $\rho$ is a contrast function.
    We now need to prove that $\rho$ induces the statistical structure $(g,\nabla)$.
    For any $X,Y,Z\in\Gamma(TM)$ and $x\in M$, we have
    \begin{equation*}
        \begin{split}
            \rho(X|Y)(p) &= -\rho(|XY)(p)\\
                        &= -e^{-\sigma(p)}(X^LY^L(\tilde{d})^2)(p,p)\\
                        &= -e^{-\sigma(p)}\tilde{g}_p(X,Y)\\
                        &= -g_p(X,Y),
        \end{split}
    \end{equation*}
    \begin{equation*}
        \begin{split}
            \rho(XY|Z)(p) &= (X^RY^RZ^L\rho)(p,p)\\
                        % &= X^RY^R(e^{-\sigma}Z^L(\tilde{d})^2)(p,p)\\
                        % &= -X^R(d\sigma(Y)e^{-\sigma}Z^L(\tilde{d})^2)(p,p) + X^R(e^{-\sigma}Y^RY^L(\tilde{d})^2)(p,p)\\
                        &= (XYe^{-\sigma})(p)(Z^L(\tilde{d})^2)(p,p) + (Ye^{-\sigma})(p)(X^RZ^L(\tilde{d})^2)(p,p)\\
                        &+ (Xe^{-\sigma})(p)(Y^RZ^L(\tilde{d})^2)(p,p) + e^{-\sigma(p)}(X^RY^RZ^L(\tilde{d})^2)(p,p)\\
                        &= d\sigma_p(Y)e^{-\sigma(p)}\tilde{g}_p(X,Z) + d\sigma_p(X)e^{-\sigma(p)}\tilde{g}_p(Y,Z) - e^{-\sigma(p)}\tilde{g}_p(\nabla^{\tilde{g}}_X Y,Z)\\
                        &= -g_p(\nabla_XY,Z)
                        % &= d\sigma_p(Z)e^{\sigma(p)}(X^RY^R(d^*)^2)(p,p) + e^{\sigma(p)}(X^RY^RZ^L(d^*)^2)(p,p)\\
                        % &= d\sigma_p(Z)e^{\sigma(p)}g^*_p(X,Y) - e^{\sigma(p)}g^*_p(\nabla^{g^*}_X Y,Z)\\
                        % &= d\sigma_p(Z)g_p(X,Y) - g_p(\nabla^{g^*}_X Y,Z)\\
                        % &= g_p(X,Y)g_p(\operatorname{grad}_g\sigma,Z) \\
                        % &- g_p\left(\nabla^g_XY - \frac{1}{2}\left(d\sigma_p(X)Y + d\sigma_p(Y)X - g_p(X,Y)\operatorname{grad}_g\sigma\right),Z\right)\\
                        % &= -g_p(\nabla_X Y,Z).
        \end{split}
    \end{equation*}
\end{proof}

\begin{remark}
    Denote $\overline{\nabla}$ the conjugate connection. 
    If $(g,\overline{\nabla})$ is a statistical structure induced by an affine immersion, this canonical contrast function of $(M,g,\nabla)$ coincides with the contrast function of $(M,g,\overline{\nabla})$ given in Example \ref{ccf1}.\\
\end{remark}
\begin{remark}
    If $(M,\tilde{g})$ is complete, for each $p,q\in M$ there is a geodesic $\tilde{\gamma}(s),\, s\in [0,1]$ of $\nabla^{\tilde{g}}$ such that $\tilde{\gamma}(0)=p$, $\tilde{\gamma}(1)=q$, and
    \begin{equation*}
        \tilde{d}(p,q) =\int_0^1 \sqrt{\tilde{g}\left(\frac{d\tilde{\gamma}}{ds},\frac{d\tilde{\gamma}}{ds}\right)}\,ds.
    \end{equation*}
    The left-hand side of this equation does not depend on the parametrization of $\tilde{\gamma}$, and the affine connection $\nabla$ is projectively equivalent to $\nabla^{\tilde{g}}$.
    Therefore, if $(M,\tilde{g})$ is complete, for each $p,q\in M$ there exists a geodesic $\gamma(t),\, t\in [0,a]$ of $\nabla$ such that $\gamma(0)=p$, $\gamma(a)=q$, and
    \begin{equation*}
        \rho(p,q) = e^{-\sigma(q)}\left(\int_0^a \sqrt{e^{\sigma\circ\gamma}g\left(\frac{d\gamma}{dt},\frac{d\gamma}{dt}\right)}\,dt\right)^2.
    \end{equation*}
\end{remark}

\section{The Hopf-Rinow type theorem on a statistical manifold with divisible cubic form}\label{sec3}
We derive the geodesic completeness of the affine connection on a statistical manifold with divisible cubic form.
We also derive the completeness of the affine connection using the contrast function.\\

\subsection{Geodesic connectedness of the affine connection}
An affine connection $\nabla$ on a manifold $M$ is called \textit{geodesically connected}, if for any $p,q\in M$ there exists a geodesic $\gamma(t),\,t\in[0,1]$ of $\nabla$ such that $\gamma(0)=p$ and $\gamma(1)=q$.
This is equivalent to the exponential map $\operatorname{exp}_p:D_p \to M$ of $\nabla$ being surjective at each $p\in M$.
Here, $D_p\subset T_p M$ is the maximal open neighborhood of $0\in T_pM$ such that the exponential map can be defined.\\
% The Hopf-Rinow type theorem in Riemannian geometry states that if the Levi-Civita connection on a Riemannian manifold is complete, then it must be geodesically connected.
% In this section, we prove that the affine connection on a statistical manifold with divisible cubic form also has a similar property.\\

\begin{lemma}
    \label{lemcom}
    Let $(M,g,\nabla)$ be a statistical manifold with divisible cubic form $C=\operatorname{sym}(d\sigma\otimes g)$ and $\tilde{g} = e^{\sigma}g$.
    \renewcommand{\labelenumi}{$\operatorname{(\theenumi)}$}
    \begin{enumerate}
        \item If $\nabla$ is complete and $\sigma$ is bounded from below, then $\nabla^{\tilde{g}}$ is also complete. \label{com1}
        \item If $\nabla^{\tilde{g}}$ is complete and $\sigma$ is bounded from above, then $\nabla$ is also complete. \label{com2}
    \end{enumerate}
\end{lemma}
\begin{proof}
    We prove $(\ref{com2})$ first. 
    For an arbitrary geodesic $\gamma(t),\,t\in[0,a)$ of $\nabla$ with $a<\infty$, we prove that $\gamma$ is extendable beyond $a$ as a geodesic.
    Parametrize $\gamma$ by $\tilde{\gamma} = \gamma\circ s^{-1}$ where $s:[0,a)\to[0,\alpha)$ is the parameter transformation defined by $(\ref{paramtrans})$.
    As shown before, $\tilde{\gamma}$ is a geodesic of $\nabla^{\tilde{g}}$.
    Since $\sigma$ is bounded from above, let $\sigma\leq L$ on $M$ for some $L\in\mathbb{R}$.
    We have
    \begin{equation*}
        \alpha = \lim_{t\to a}\int_0^t e^{2\sigma(\gamma(r))} dr \leq a\cdot e^{2L},
    \end{equation*}
    hence $\alpha<\infty$. Since $\nabla^{\tilde{g}}$ is complete, $\tilde{\gamma}$ is extendable beyond $\alpha$ as a geodesic.
    Therefore, if we reparametrize the extended geodesic of $\nabla^{\tilde{g}}$ back to the geodesic of $\nabla$, we obtain the extension of $\gamma$ beyond $a$ as a geodesic.
    Statement $(\ref{com1})$ is proved in the same manner.\\
\end{proof}

\begin{proposition}
    \label{main1}
    Let $(M,g,\nabla)$ be a statistical manifold with divisible cubic form $C=\operatorname{sym}(d\sigma\otimes g)$ and $\tilde{g} = e^{\sigma}g$.
    If $\nabla^{\tilde{g}}$ is complete, then $\nabla$ is geodesically connected.\\
\end{proposition}
\begin{proof}
    Let $\nabla^{\tilde{g}}$ be complete, and $p,q\in M$ be arbitrary points. 
    From the classical Hopf-Rinow theorem, there exists a geodesic $\tilde{\gamma}(s),\,s\in[0,1]$ of $\nabla^{\tilde{g}}$ such that $\tilde{\gamma}(0) = p$ and $\tilde{\gamma}(1) = q$.
    Define a parameter transformation $t:[0,1]\to[0,a]$ by
    \begin{equation*}
        t(s) = \int_0^s e^{-2\sigma(\tilde{\gamma}(r))} dr,\quad s\in[0,1).
    \end{equation*}
    In the same manner as the parameter transformation $(\ref{paramtrans})$, the curve $\gamma=\tilde{\gamma}\circ t^{-1}$ is a geodesic of $\nabla$.
    Therefore, we obtain a geodesic $\gamma(t),\,t\in[0,a]$ of $\nabla$, such that $\gamma(0)=p$ and $\gamma(a)=q$.\\
\end{proof}

\begin{example}
    Consider the statistical structure $(g,\nabla)$ on $\mathbb{R}^2$ obtained from the centroaffine immersion in Example \ref{ellipex}.
    The affine connection $\nabla$ is complete, according to Example 5.6 in \cite{BO}.
    Indeed, the Riemannian metric $g^*=e^{-\sigma}g$ is the Euclidean metric on $\mathbb{R}^2$, and the smooth function $-\sigma$ is bounded from below.
    The affine connections $\nabla^{g^*}$ and $\nabla$ are related by
    \begin{equation}
        \label{rel}
        \nabla_X Y = \nabla^{g^*}_X Y - g^*(X,Y)\operatorname{grad}_{g^*}\sigma, \quad X,Y\in\Gamma(TM).
    \end{equation}
    The Levi-Civita connection $\nabla^{g^*}$ is complete, and by Proposition \ref{compconns} the affine connection $\nabla$ is complete.
    By Proposition \ref{main1}, we also have that the conjugate connection $\overline{\nabla}$ is geodesically connected.
    Also, note that $\overline{\nabla}$ is not complete, thus the condition of $\sigma$ being bounded from above is needed in lemma \ref{lemcom}.
    Indeed, since $(M,g,\overline{\nabla})$ is conjugate symmetric, from equation $(\ref{conjcovc})$ the cubic form $\overline{C}$ satisfies
    \begin{equation*}
        (\nabla^g_X \overline{C})(X,X,X) = -\frac{3}{2}\Delta_g\sigma\cdot\left(g(X,X)\right)^2\geq0, \quad X\in\Gamma(TM).
    \end{equation*}
    By Theorem 4.1 of \cite{BO}, the completeness of the affine connection $\overline{\nabla}$ implies that $(g,\overline{\nabla})$ is Riemannian, which is a contradiction. \\
\end{example}

% With Proposition \ref{main1}, we obtain the Hopf-Rinow type theorem on statistical manifolds with divisible cubic forms.\\
\begin{theorem}
    \label{main2}
    Let $(M,g,\nabla)$ be a statistical manifold with divisible cubic form of $C=\operatorname{sym}(d\sigma\otimes g)$, and assume that $\sigma$ is bounded from below.
    \renewcommand{\labelenumi}{$\operatorname{(\theenumi)}$}
    \begin{enumerate}
        \item If $\nabla$ is complete, then $\nabla$ is geodesically connected. \label{geocon1}
        \item If $\nabla^{g}$ is complete, then $\nabla$ is geodesically connected. \label{geocon2}
    \end{enumerate}
\end{theorem}
\begin{proof}
    We prove $(\ref{geocon1})$ first. 
    Define a Riemannian metric $\tilde{g} = e^{\sigma}g$ on $M$.
    By Lemma \ref{lemcom}, we have that $\nabla^{\tilde{g}}$ is complete.
    Therefore, by Proposition \ref{main1} $\nabla$ is geodesically connected.\\
    For $(\ref{geocon2})$ as well, it follows from equation $(\ref{contrans})$ and Proposition \ref{compconns} that $\nabla^{\tilde{g}}$ is complete.\\
\end{proof}

The boundedness of $\sigma$ from below is necessary in Theorem \ref{main2}.\\
\begin{example}
    \label{countex}
    Let $M=\mathbb{R}^2-\{(0,0)\}$ and $g_0$ the Euclidean metric. 
    By \textrm{\cite{04b7b8b8-0f3e-3608-9e2a-b8803be43fe8}}, there exist complete Riemannian metrics on $M$ which are conformal to $g_0$.
    For example, consider the function $\sigma(x^1,x^2)=\frac{2}{(x^1)^2+(x^2)^2}, (x^1,x^2)\in M$ and define $g=e^{\sigma}g_0$.
    Then, $\nabla^g$ is a complete affine connection on $M$. 
    With an affine connection given by the following equation, $(M,g,\nabla)$ is a statistical manifold with divisible cubic form $C=\operatorname{sym}(-d\sigma\otimes g)$$:$
    \begin{equation*}
        \nabla_X Y = \nabla^{g_0}_X Y + d\sigma(X)Y + d\sigma(Y)X, \quad X,Y\in\Gamma(TM).
    \end{equation*}
    The function $-\sigma$ is bounded from above on $M$, thus $\nabla$ is complete from Proposition \ref{compconns}.
    Therefore, on $(M,g,\nabla)$ the affine connections $\nabla$ and $\nabla^g$ are both complete.
    However, the image of any geodesic of $\nabla$ is a straight line in $M$, since $\nabla$ is projectively equivalent to $\nabla^{g_0}$.
    It is obvious that $\nabla$ is not geodesically connected on $M$.\\
\end{example}

\subsection{Completeness of the affine connection and the contrast function}
Since contrast functions do not necessarily satisfy the triangle inequality, it seems difficult to achieve the original Hopf-Rinow theorem completely.
However, there is a way to derive the completeness of the affine connection using the canonical contrast function.\\
% \begin{lemma}
%     \label{lemhopf}
%     Let $(M,g,\nabla)$ be a statistical manifold with divisible cubic form of $C=\operatorname{sym}(d\sigma\otimes g)$.
%     Define $g^*=e^{-\sigma}g$ and assume that $\sigma$ is bounded from above.
%     If $\nabla^{g^*}$ is complete, then $\nabla$ is also complete.\\
% \end{lemma}
% \begin{proof}
%     This follows from Proposition \ref{compconns}, since the affine connections $\nabla$ and $\nabla^{g^*}$ are related as in equation $(\ref{rel})$.
% \end{proof}

% We introduce a notion of a bounded set on a Manifold equipped with a contrast function.\\
\begin{definition}
    Let $\rho$ be a contrast function on a manifold $M$.
    The pair $(M,\rho)$ is said to be \textit{sublevel set complete} if there exists $p\in M$ such that for any $L>0$, the sublevel set
    \begin{equation*}
        C_L(p) = \{q\in M\,|\,\rho(p,q)\leq L\}
    \end{equation*}
    is compact in $M$.\\
\end{definition}

If $\rho$ is a distance function on $M$, then $(M,\rho)$ is a complete metric space if and only if it is sublevel set complete.\\

% \begin{theorem}
%     \label{main3}
%     Let $(M,g,\nabla)$ be a statistical manifold with divisible cubic form of $C=\operatorname{sym}(d\sigma\otimes g)$ and $\rho$ the contrast function given in Theorem $\operatorname{\ref{contfunc}}$.
%     Assume that $\sigma$ is bounded from above.
%     If any closed bounded set of $(M,\rho)$ is compact, then $\nabla$ is complete.\\
% \end{theorem}
\begin{theorem}
    \label{main3}
    Let $(M,g,\nabla)$ be a statistical manifold with divisible cubic form of $C=\operatorname{sym}(d\sigma\otimes g)$ and $\rho$ the contrast function given in Theorem $\operatorname{\ref{contfunc}}$.
    \renewcommand{\labelenumi}{$\operatorname{(\theenumi)}$}
    \begin{enumerate}
        \item If $(M,\rho)$ is sublevel set complete and $\sigma$ is bounded from above, then $\nabla$ is complete. \label{com3}
        \item If $\nabla$ is complete and $\sigma$ is bounded from below, then $(M,\rho)$ is sublevel set complete. \label{com4}
        \item If $(M,\rho)$ is sublevel set complete and $\sigma$ is bounded from below, then $\overline{\nabla}$ is complete. \label{com5}
    \end{enumerate}
\end{theorem}

\begin{proof}
    For any $p\in M$ and $L>0$, the closed set $C_L(p)$ is equal to the set
    \begin{equation*}
        \left\{q\in M\,|\,\tilde{d}(p,q)\leq \sqrt{e^{\sigma(p)}L}\right\}.
    \end{equation*}
    Here, $\tilde{d}$ is the distance function of $\tilde{g}=e^{\sigma}g$.
    Thus, $(M,\rho)$ is sublevel set complete if and only if the metric space $(M,\tilde{d})$ is complete.
    Lemma \ref{lemcom} completes the proof for $(\ref{com3})$ and $(\ref{com4})$.
    For $(\ref{com5})$, the connections $\nabla^{\tilde{g}}$ and $\overline{\nabla}$ are related as follows$:$
    \begin{equation*}
        \overline{\nabla}_X Y = \nabla^{\tilde{g}}_X Y + \tilde{g}(X,Y)\operatorname{grad}_{\tilde{g}}\sigma,\quad X,Y\in\Gamma(TM).
    \end{equation*}
    Proposition \ref{compconns} completes the proof.\\
\end{proof}
% \begin{proof}
%     Define $g^*=e^{-\sigma}g$ on $M$, and let $d^*$ be the distance function of $(M,g^*)$.
%     Since $\sigma$ is bounded from above, take an $L\in\mathbb{R}$ such that $e^{\sigma}\leq L$ on $M$.
%     Fix an arbitrary closed bounded set $A\subset M$ of the metric space $(M,d^*)$ and define $L_A = \operatorname{sup}\{d^*(p,q)\,|\,p,q\in A\}$.
%     Then, for any $p,q\in A$, we have
%     \begin{equation*}
%         \rho(p,q) = e^{\sigma(q)}(d^*(p,q))^2 < L\cdot L_A^2.
%     \end{equation*}
%     Thus, $A$ must be a compact set. 
%     It follows that $(M,d^*)$ is a complete metric space, and $\nabla^{g^*}$ is complete.
%     Therefore, from Lemma \ref{lemhopf} the affine connection $\nabla$ is complete.\\
% \end{proof}

\begin{example}
    Consider an ellipsoid $f:M^n\to\mathbb{R}^{n+1}$, that is, a locally strongly convex centroaffine immersion from a compact manifold $M$ such that the image is an ellipsoid.
    With the statistical structure $(g,\nabla)$ induced on $M$, the statistical manifold $(M,g,\nabla)$ has divisible cubic form.
    From Theorem \ref{main3}, the affine connection $\nabla$ is complete. Moreover, from Theorem \ref{main2}, it is geodesically connected.\\
\end{example}

\section{The Cartan-Hadamard type theorem on a statistical manifold with divisible cubic form}
If an affine connection is geodesically connected, the next step is to analyze how many geodesics connect the same two points on the manifold. 
That is, we examine the injectivity of the exponential map at each point.
The following Cartan-Hadamard type theorem on statistical manifolds with divisible cubic forms provides a condition that ensures any two points on the manifold are connected by a unique geodesic.\\
\begin{theorem}
    \label{cartanh}
    Let $(M,g,\nabla)$ be a statistical manifold with divisible cubic form of $C=\operatorname{sym}(d\sigma\otimes g)$.
    Assume that $\sigma$ is bounded from below and $\nabla$ is complete.
    If for each pair of orthonormal vectors $X,Y$,
    \begin{equation}
        \label{careq}
        2g\left(S(X,Y)Y,X\right) - \left((\nabla^g_Y d\sigma)(Y) + (\nabla^g_X d\sigma)(X) + \|d\sigma\|^2_g\right)\leq 0
    \end{equation}
    holds for the statistical curvature tensor $S$ of $(M,g,\nabla)$.
    Then, for any $p$ the exponential map $\operatorname{exp}_p:T_pM\to M$ of $\nabla$ is a covering map.
    In particular, if $M$ is simply connected, then $\operatorname{exp}_p$ is a diffeomorphism at each $p\in M$.\\
\end{theorem}

\begin{proof}
    First, we will prove the theorem under the assumption that $M$ is simply connected.
    From Lemma \ref{lemcom}, the Levi-Civita connection $\nabla^{\tilde{g}}$ is complete, where $\tilde{g}=e^{\sigma}g$.
    Let $\tilde{R}$ be the curvature tensor field of $\nabla^{\tilde{g}}$.
    By substituting equation $(\ref{careq})$ into equation $(\ref{seccurvtilde})$, the complete Riemannian manifold $(M,\tilde{g})$ has nonpositive sectional curvature.
    Hence, from the classical Cartan-Hadamard theorem, the exponential map $\bwidetilde{\operatorname{exp}}_p:T_pM\to M$ of $\nabla^{\tilde{g}}$ at each $p\in M$ is a diffeomorphism.
    Now, from Theorem $\ref{main2}$, the affine connection $\nabla$ is geodesically connected, so the exponential map $\operatorname{exp}_p:T_pM\to M$ of $\nabla$ is surjective at each $p\in M$.
    We will prove that the exponential maps are also injective.
    Assume that two vectors $v_1,v_2\in T_pM$ satisfies $\operatorname{exp}_p(v_1)=\operatorname{exp}_p(v_2)=q$.
    Set $\gamma_i(t) = \operatorname{exp}_p(tv_i),\,t\in[0,1]$ for each $i=1,2$.
    These curves are geodesics of $\nabla$ that connects the points $p$ and $q$.
    Hence, in the same manner as the parameter transformation $(\ref{paramtrans})$, the geodesics $\gamma_1,\gamma_2$ can be reparameterized into geodesics of $\nabla^{\tilde{g}}$.
    Since the exponential maps of $\nabla^{\tilde{g}}$ are diffeomorphisms, the geodesics $\gamma_1,\gamma_2$ must coincide, implying that $v_1=v_2$.\\

    Next, we will prove the theorem for the case when $M$ is not necessarily simply connected.
    Let $\varpi:M'\to M$ be the universal covering of $M$.
    Since $\varpi$ is an immersion, we define a Riemannian metric $g'$ and an affine connection $\nabla'$ by $g' = \varpi^*\tilde{g}$ and
    \begin{equation*}
        \varpi_*\nabla'_X Y = \nabla_{\varpi_*X} \varpi_*Y, \quad X,Y\in\Gamma(TM').
    \end{equation*}
    The affine connections $\nabla'$ and $\nabla^{g'}$ are both complete on $M'$.
    Indeed, if $\gamma(t),\,\in[0,a)$ is a geodesic of $\nabla'$ with $a<\infty$, then $\varpi\circ\gamma$ is a geodesic of $\nabla$.
    The affine connection $\nabla$ is complete, so the curve $\varpi\circ\gamma$ can be extended beyond $a$ as a geodesic of $\nabla$.
    Since $\varpi$ is locally a diffeomorphism, the curve $\gamma$ can be extended beyond $a$ as a geodesic of $\nabla'$.
    From the classical Hopf-Rinow theorem, we know that $\nabla^{g'}$ is geodesically connected on $M'$.
    Define a 1-form $\omega$ on $M'$ by $\omega = \varpi^*d\sigma$.
    The 1-form $\omega$ is an exact form since $\omega=d(\sigma\circ\varpi)$.
    We also have
    \begin{equation*}
        \nabla^{g'}_X Y = \nabla'_X Y + \omega(X)Y + \omega(Y)X,\quad X,Y\in\Gamma(TM'),
    \end{equation*}
    thus, the affine connections $\nabla^{g'}$ and $\nabla'$ are projectively equivalent, and $\nabla'$ is also geodesically connected on $M'$.
    By following the first part of this proof, we conclude that the exponential map $\operatorname{exp}_p':T_p M'\to M'$ of $\nabla'$ is a diffeomorphism at each $p\in M'$.
    The following diagram for the exponential map $\operatorname{exp}_p$ of $\nabla$ on $p\in M$, where $\varpi(p')=p$ commutes$:$
    \[
    \begin{CD}
     T_{p'}M' @>{\operatorname{exp}_{p'}'}>> M' \\
    @V{\varpi_*}VV    @V{\varpi}VV \\
     T_p M   @>{\operatorname{exp}_{p}}>>  M
    \end{CD}
    \]
    Thus, the mapping $\varpi_*\circ\operatorname{exp}_{p}:T_{p'}M\to M$ is a covering map.
    Since $\varpi_*$ is an isomorpshim, the exponential map $\operatorname{exp}_p:T_pM\to M$ of $\nabla$ is a covering map at each $p\in M$.\\
    \end{proof}

    \begin{cor}
        \label{corch}
        Let $(M,g,\nabla)$ be a statistical manifold with divisible cubic form of $C=\operatorname{sym}(d\sigma\otimes g)$.
        Assume that $(M,g)$ is a $2$ dimensional Riemannian manifold, the function $\sigma$ is bounded from below, and $\nabla$ is complete.
        If the inequality
        \begin{equation}
            \label{careq2}
            2k^g\leq \Delta_g\sigma
        \end{equation}
        holds for the sectional curvature $k^g$ of $(M,g)$, then for any $p$ the exponential map $\operatorname{exp}_p:T_pM\to M$ of $\nabla$ is a covering map.\\
    \end{cor}
    \begin{proof}
        Let $k^g$ be the sectional curvature of $(M,g)$.
        The statistical curvature tensor field of $(M,g,\nabla)$ is
        \begin{equation*}
            \begin{split}
                S(X,Y)Z &= R^{g}(X,Y)Z + \frac{1}{2}\left(d\sigma(Y)d\sigma(Z)X-d\sigma(X)d\sigma(Z)Y\right)\\
                &+\frac{1}{2}\left(g(Y,Z)(\|d\sigma\|_g^2X+d\sigma(X)\operatorname{grad}_g\sigma)-g(X,Z)(\|d\sigma\|_g^2Y+d\sigma(Y)\operatorname{grad}_g\sigma)\right),
            \end{split}
        \end{equation*}
        for $X,Y,Z\in\Gamma(TM)$.
        Thus, from the assumption $(\ref{careq2})$, for orthonormal frames $(X,Y)$ on $(M,g)$  we have
        \begin{equation*}
            \begin{split}
                2g(S(X,Y)Y,X) &= 2k^g + \left(d\sigma(X)d\sigma(X)+d\sigma(Y)d\sigma(Y)+\|d\sigma\|_g^2\right)\\
                &= 2k^g + \|d\sigma\|_g^2\\
                &\leq \|d\sigma\|_g^2 + \Delta_g\sigma\\
                &= \|d\sigma\|_g^2 + (\nabla^g_Xd\sigma)(X) + (\nabla^g_Yd\sigma)(Y).
            \end{split}
        \end{equation*}
        This implies that the equation $(\ref{careq})$ is satisfied on the statistical manifold $(M,g,\nabla)$.
        Therefore, Theorem \ref{cartanh} completes this proof.
    \end{proof}

    \begin{example}
        Let $(\mathbb{H}^2,g)$ be the upper half plane with the standard Riemannian metric.
        With a bounded function $\sigma(x,y) = e^{-y}, (x,y)\in\mathbb{H}^2$ on $\mathbb{H}^2$, we obtain a statistical manifold $(\mathbb{H}^2,g,\nabla)$ with
        \begin{equation*}
            \nabla = \nabla^{g} + K
        \end{equation*}
        where $K$ is the $(1,2)$-tensor defined by $(\ref{difften})$. 
        Since $\sigma$ is a bounded function on the complete Riemannian manifold ($\mathbb{H}^2$,g), by Proposition \ref{compconns} $\nabla$ is complete.
        For any $(x,y)\in\mathbb{H}^2$, we have 
        \begin{equation*}
            (\|d\sigma\|_g(x,y))^2 = y^2e^{-2y},
        \end{equation*}
        and
        \begin{equation*}
            (\Delta_g\sigma)(x,y) = y^2e^{-y}.
        \end{equation*}
        By Corollary \ref{corch}, the exponential map $\operatorname{exp}_p:T_p\mathbb{H}^2\to\mathbb{H}^2$ of $\nabla$ is a diffeomorphism at each $p\in\mathbb{H}^2$.
    \end{example}

\subsection*{Acknowledgments}
The author would like to express their gratitude to Professor Hitoshi Furuhata and Dr. Hirotaka Kiyohara for their valuable support and insightful discussions during the course of this research.

\bmhead{Funding}
This work was supported by the Japan Science and Technology agency, the establishment of university fellowships towards the creation of science technology innovation, Grant Number JPMJSP2119.

% \bmhead{Data availability}
% No data was gathered for this article.

% \bmhead{Declarations}

% \bmhead{Conflict of interest}

\bibliography{sn-bibliography}% common bib file

\end{document}